\documentclass[a4paper,english,12pt]{amsart}
\usepackage[all]{xy}
\usepackage{frcursive}
\usepackage{eucal}
\usepackage{graphicx}
\usepackage{epsfig}
\usepackage{mathrsfs}
\usepackage{amssymb}
\usepackage{amsxtra}
\usepackage{enumerate}
\newcommand{\deactivateaddcontentsline}{%
  \let\addcontentslineoriginal\addcontentsline
  \renewcommand\addcontentsline[3]{}%
}
\newcommand{\reactivateaddcontentsline}{%
  \let\addcontentsline\addcontentslineoriginal
}

\input cyracc.def 
 
\newtheorem{theorem}{Theorem}
\newtheorem{proposition}{ Proposition}
\newtheorem{lemma}{ Lemma}
\newtheorem{corollary}{Corollary}

\newtheorem{definition}{Definition}

\theoremstyle{remark}

\def \1{\mathbb {1}}
\def \RM{\mathbb {R}}
\def \NM{\mathbb{N}}
\def \ZM{\mathbb{Z}}

 \def \Vol {{\rm Vol}}

\def \p {{\rm exp\,}}

\def \d{\partial}
\def\dt{\delta} 
\def\a{\alpha}
\def\b{\beta}
\def\e{\varepsilon}  
\def\g{\gamma}

\def\l{\lambda}
\def\L{\Lambda}
\def\p{\varphi}  
\def\lb{\left\{}
\def\rb{\right\}}
\def\G{\Gamma}   

\def \s{\sigma}

\def \to{\longrightarrow} 
\def \w{\wedge}

\def \< {{\thengle }}
\def \> {{\rangle }}
\def \( {\left( }
\def \) {\right) }

\newcommand{\Ct}{{\mathcal C}}

\newcommand{\Lt}{{\mathcal L}}
\title[Arithmetic density]{Arithmetic density}
\author{  Mauricio  Garay}
\address{ Institut für Mathematik\\
FB 08 - Physik, Mathematik und Informatik\\
Johannes Gutenberg-Universität Mainz\\
Staudinger Weg 9\\
55128 Mainz.}
\begin{document}
\begin{abstract}{ We define arithmetic classes as subsets of $\RM^n$ consisting of vectors which are not better approximated by the integer lattice than a given sequence. We give measure estimates for such sets.}
\end{abstract}

\maketitle

\deactivateaddcontentsline
\section{Introduction}
In the study of dynamical systems, one frequently consider deformations over completely discontinuous subsets.  The KAM theorem provides a typical example, where invariant tori of a hamiltonian system are parametrised by diophantine frequency vectors~\cite{Arnold_KAM,Kolmogorov_KAM,Moser_KAM}. 

Such vectors are defined by the existence of constants $C,\tau$ such that
$$| (v,i)| \geq \frac{C}{\| i \|^\tau},\ \forall i \in \ZM^n \setminus \{ 0 \} . $$
For fixed $\tau$, one can easily prove that the set 
$$\Omega_\tau=\{ v \in \RM^n: \exists C, \forall i \in \ZM^n \setminus \{ 0 \}, | (v,i)| \geq \frac{C}{\| i \|^\tau} \} $$
is of full measure. Invariant tori are parametrised by  the preimage of some $\Omega_{\tau}$ by a map~:
$$f:\RM^d \to \RM^n ,$$
called the {\em frequency mapping}.

In the original situation considered by Kolmogorov, the map $f$ was a local diffeomorphism so that the preimage by $f$ of $\Omega_{\tau}$ is also a full-measure set. In the late sixties, Arnold and Pyartli  gave sufficient conditions,  in case the image of  $f$ defines a submanifold in $\RM^n$, to ensure that $f^{-1}(\Omega_\tau)$ is still a full measure set~\cite{Pyartli}. This result led Rüssmann to a generalisation of the KAM theorem~\cite{Russmann_KAM}. In a series of works, Dani, Kleinbock, Margulis et al. introduced discrete group theoretic techniques in diophantine approximation giving major improvements of the Arnold-Pyartli theorem~\cite{Dani,Kleinbock,Kleinbock_Margulis,Margulis}. In this paper, we show that these techniques are efficient
not only for studying the sets $\Omega_\tau$ but also its closed subsets
 $$\Omega_{C,\tau}=\{ v \in \RM^n: \forall i \in \ZM^n \setminus \{ 0 \}, | (v,i)| \geq \frac{C}{\| i \|^\tau} \} $$
 for fixed $C>0$. 
 
In KAM theory,  this constant $C$ is related to the perturbation parameter and this relation leads to intricate estimates. It seems, however, difficult to fix a constant $C$ since  some of the subsets  $\Omega_{C,\tau}$ might have locally zero measure and even contain isolated segments. Hence, measure estimates for $\Omega_{C,\tau}$'s appear, at first glance, impossible to obtain. 

However, if we compare the sets $\Omega_{C,\tau}$ between themselves then one can prove that, for $\tau'$ big enough, the set  $\Omega_{C,\tau'}$ has density one at the points of $\Omega_{C,\tau}$. This is the starting observation of this paper. Then, it turns out that the group theoretic techniques of diophantine analysis can be efficiently applied in the local situation, so one may show that this property is preserved by a large class of mappings. This provides the density estimate which forms the subject of this paper. Applications of our theorem to hamiltonian dynamics will be published elsewhere.

\section{Statement of the theorem}
We denote by $U \subset \RM^d$ an open neighbourhood of the origin.
\begin{definition}[\cite{Kleinbock,Kleinbock_Margulis,Pyartli}] A $C^l$-map
$$f:U \to \RM^n,\ x=(x_1,\dots,x_d) \mapsto (f_1(x),\dots,f_n(x))$$
is called $l$-curved in a vector subspace $V \subset \RM^n$ at a point $x \in \RM^d$  if~:
\begin{enumerate} [{\rm i)}]
\item $f(U)=f(x)+V$ ;
\item the partial derivatives $\{ \d^j f(x): |j | \leq l \}$ generate the vector-space $V$ where $j=(j_1,j_2\dots, j_d)$ is a multi-index and $|j |=j_1+j_2+\dots+j_d$.
\end{enumerate}
\end{definition}
Such mappings are also called {\em non-degenerate}. Sometimes we omit to specify $V$ or $l$.  

Let us now define the subsets of $\RM^n$ that we wish to consider. Denote by $(\cdot,\cdot)$ the euclidean scalar product in $\RM^n$.  For any vector $\a \in \RM^n $, we define the sequence $\s(\a)$ by~:
$$\s(\a)_k :=\min \{ |(\a,i)|: i \in \ZM^n \setminus \{ 0 \}, \| i \| \leq 2^k \} .$$
\begin{definition} The arithmetical class in $\RM^n$ associated to a real decreasing sequence $a=(a_k)$ is the set
$$\Ct(a):=\{\a \in \RM^n: \s(\a)_k  \geq  a_k\} .$$
\end{definition}
Arithmetical classes are closed subsets in $\RM^n$ with the property~:
$$\a \in \Ct(a) \implies \l \a \in \Ct(a),\ \forall \l>1 $$

The arithmetical class in $\RM^n$ associated to the sequence $(C2^{-\tau n})$ is the set $\Omega_{C,\tau}$ of $(C,\tau)$-diophantine vectors.
 By Dirichlet's theorem, for any $C>0,\ \tau \leq n$, it is empty~\cite{Dirichlet} (see also \cite{Cassels}). On the other extreme, for $\tau > n$, the union of arithmetical classes associated to the sequence $(C2^{-\tau n})$  over different values of $C$ defines a set of full measure. Similar dichotomy holds for submanifolds of $\RM^n$~\cite{Kleinbock_dichotomy}.

If instead of a countable union, we fix {\em one} arithmetical class then it is a closed subset. Therefore it cannot be of full measure unless it is equal to $\RM^n$ itself. The theorem below states that it is nevertheless of positive measure near some points and that this property is preserved
by  curved mappings. 

Let us  recall the definition of density, in the measure theoretical sense.  For $\a \in \RM^n$, we denote by $B(\a,r)$ the ball centred at $\a$ with radius~$r$. The {\em density} of a measurable subset $K \subset \RM^n$ at a point $\a$ is the limit~(if it exists)~:
$$\lim_{r \to 0} \frac{\Vol(K \cap B(\a,r))}{\Vol(B(\a,r))}.$$

The density of a measurable subset is equal to $1$ at almost all of its point~\cite{Lebesgue_derivation}. For instance, sets of zero Lebesgue measure have density equal to one at {\em almost all points} and, in fact, equal to zero at {\em all points}.

 \begin{theorem}
\label{T::arithmetique}
Consider a real positive decreasing sequence $a=(a_k)$. Define the sequence $ a'=(a_k')$ by
$$a_k':= 2^{-kn-kl(n+1)-(n+1)^2 kdl}a_k^{(n+1)l} $$
 For any $l$-curved mapping
 $$f=(f_1,\dots,f_n):U \to \RM^n,\  f(0) =\a \in \Ct(a) $$
 the density of the set $f^{-1}(\Ct(a'))$ at the origin is equal to $1$. 
\end{theorem}

 \section{Functions of class $(C,\tau)$.}
 \label{SS::Calpha}
 
Our first aim consists in proving the
\begin{proposition}
   \label{P::KM}
    For any  $l$-curved mapping
   $$f=(f_1,\dots,f_n):B(0,3R) \subset U \to \RM^n, f(0) \neq 0  $$
   there exist constants $A,C,r_0 >0$   such that
   $$\Vol(\{ x \in B(0,r): |(f(x),i)| \leq a \}) \leq C \left(a \| i \|^n \right)^{1/dl(n+1)} r^{-1/d}\Vol(B(0,r)) $$
 for any $r \leq r_0$ and any $a \in \RM_+,i \in \ZM^n \setminus \{ 0 \}$ which satisfy $$\left\{ \begin{matrix} \left(a\| i \|^n\right)^{\frac{1}{n+1}}& \leq &Ar^l\ ; \\  a & \leq & \| i \|.\end{matrix} \right. $$
 \end{proposition}

 For a subset $K \subset \RM^d$ and a function  
 $$f:K \to \RM$$
 we define
 $$  \| f \|_K:=\sup_{x \in K} | f(x) |$$
 (which might be infinite) and use the conventions $1/0=+\infty$, $1/+\infty=0$. In the sequel, we denote by $U \subset \RM^d$ an open neighbourhood of the origin.
   \begin{definition}[\cite{Kleinbock_Margulis}] A map $f:U \to \RM $
  is of $(C,\tau)$-class if for any open ball $B \subset U$ and any $\e>0$, the following estimate holds~:
 $$\Vol(\{x \in B: | f(x) | \leq \e \}) \leq C \left( \frac{\e}{\| f \|_B}\right)^\tau \Vol(B). $$
 \end{definition}
 Functions of class $(C,\tau)$ define a cone~: if $f$ is of $(C,\tau)$-class then so is $\l\, f $ for any $\l \in \RM$.
 
Lets us denote by $x_1,x_2,\dots,x_d$ the coordinates in $\RM^d$. We shall use multi-index notations
$$\d^\b:=\d_{x_1}^{\b_1} \d_{x_2}^{\b_2}  \dots \d_{x_d}^{\b_d}  $$
and put $| \b |=\b_1+\b_2+\dots+\b_d$.

A compact  $K \subset \RM^d$ will be called a {\em hypercube} if it is of the type
$$K:=[a_1,a_1+\dt] \times [a_n,a_n+\dt] \times \dots \times [a_d,a_d+\dt].$$
for some real numbers $a_1,a_2,\dots,a_d,\dt$ with $\dt$ positive. The volume of such a subset is $\dt^d$.

 \begin{lemma}[\cite{Kleinbock_Margulis}] \label{L::KM}  Let  $f:U \to \RM$ be a $C^l$ function. Assume that there exists $M,m>0$ such that for any multi-index $\b$ with $|\b| \leq l$, we have~:
  \begin{enumerate}[{\rm i)}]
 \item $ \inf_{x \in U} \| \d^l_{x_i} f(x)   \| >m$ for $i=1,\dots,d$ ;
\item  $   \sup_{x \in U} \| \d^\b f(x) \| < M $. 
\end{enumerate}
 For any hypercube $K$ contained in $U$, we have
 $$\Vol(\{x \in K: | f(x) | \leq \e \}) \leq C \left( \frac{\e}{\| f \|_K}\right)^{1/dl} \Vol(K) $$
with
$$C:=dl(l+1)\left( \frac{M}{m}(l+1)(2l^l+1)\right)^{1/l}.$$
 \end{lemma}
 \begin{corollary} 
 Let $f:U \to \RM$  be a $C^l$ function. Assume that  the $l$-th order Taylor expansion of $f$ at the origin is not constant. Then
 there exist a neighbourhood of the origin and constants $C,\tau$ such that the restriction of $f$ to this neighbourhood  is of $(C,\tau)$-class.
\end{corollary}
\begin{proof}
 Indeed, up to a rotation, we may assume that the Taylor expansion of $f$ at the origin is of the type
 $$f(x)=\sum_{| I | = k} a_I x^I+o(| x|^k),\ k \leq l $$
 with $ \d^k_i  f(0) \neq 0 $ for all $i=1,\dots,d$.

 Choose $r$ sufficiently small so that there exists $m,M$ with
 $$  \| \d^k_i  f(x) \|  \geq m,\ \forall x \in B(0,r), \forall i=1,\dots,d.$$
Now, any ball with radius $\rho$ lying inside $B(0,r/\sqrt{d})$ is contained in a circunscribed hypercube whose sides have length $2\rho$, itself contained inside the ball $B(0,r)$. Thus, the estimate of the previous lemma gives constants $C,\tau$ for which the restriction of $f$ to $B(0,r/\sqrt{d})$ is of $(C,\tau)$ class.
\end{proof}
 \begin{corollary}\cite[Corollary 3.2]{Kleinbock} \cite [Proposition 3.4]{Kleinbock_Margulis}
  \label{C::KM} For any  curved mapping  
  $$f=(f_1,\dots,f_n):U \to \RM^n$$ and any point $x \in U$, there exist a neighbourhood $U'$ of $x$ and constants $C,\tau$ such that for any $c=(c_0,c_1,\dots,c_n) \in \RM^{n+1} $  the restriction of the function $c_0+c_1f_1+c_2f_2+\dots+c_nf_n$ to $U'$ is of $(C,\tau)$-class.
\end{corollary}
\begin{proof} Let $e_1,\dots,e_r \in \RM^n$ be a basis of a vector subspace $V$ so that $f$ is non-degenerate in $V$. Write
$$f=f(x)+\sum_{i=1}^r g_i e_i .$$
As $f$ is non-degenerate,  for any $c=(c_0,c_1,\dots,c_n) \in \RM^{n+1} $,  the function 
$$c_0+c_1f_1+c_2f_2+\dots+c_nf_n= c_0+(c_1,e_1) g_1+(c_2,e_2) g_2+\dots+(c_n,e_n)g_n$$ has a non-zero Taylor expansion at the origin,
provided that not all $(c_i,e_i)'s$ vanish for $i>0$. 
As the functions of $(C,\tau)$-class define a cone, it is sufficient to consider the case $\|c \|=1$. Then, repeating the proof of the previous corollary, by a uniform estimate for the derivatives, we find $C,\tau$ independent on the choice on $c$ such that all linear combinations of $1,f_1,\dots,f_n$ are of $(C,\tau)$-class.
\end{proof} 
 \section{The Kleinbock-Margulis theorem}
Denote by $e_1,e_2,\dots,e_{n+1}$ the standard basis of the vector space $\RM^{n+1}$.   
For $i=(i_1,i_2,\dots,i_k)$, $i_j < i_{j+1}$, we put
$$e_i:=e_{i_1} \w e_{i_2} \w \dots \w e_{i_k}$$
and endow the exterior algebra   $\L^\bullet \RM^n$ of a scalar product as follows.
First define the {\em Hodge operator} 
$$*:\L^p \RM^n \to \L^{n-p} \RM^n$$ by the condition
$$* u \w u=e_1 \w e_2 \w \dots \w e_n$$
and the scalar product in $\L^\bullet \RM^n$ by
$$ * u \w v=(u,v)  e_1 \w e_2 \w \dots \w e_n. $$
This endows the exterior algebra of an euclidean structure for which the $e_i$'s define an orthonormal basis. 
 
 The map 
 $$\L^\bullet \RM^n \to \L^\bullet \RM^n,\ v \mapsto -v $$
defines an action of the group  $\ZM/2\ZM$ on the exterior algebra $\L^\bullet \RM^n$. There is a well-defined injective map 
$$\G \mapsto \overline{u_1 \w u_2 \w \cdots \w u_r} $$
which sends a discrete subgroup $\G$ generated  by $u_1,u_2,\dots,u_r$ to the class $u_1 \w u_2 \w \cdots \w u_r$ in the quotient space
$\L^\bullet \RM^n/(\ZM/2\ZM) $.

We define the "{\em norm}" of the discrete subgroup $\G \subset \RM^n$ by~:
 $$\| \G \|:=\| u_1 \w u_2 \w \cdots \w u_r  \| $$

 A discrete subgroup is called   {\em primitive} if it is not a proper subgroup of a discrete subgroup with the same rank.
We denote by $\Lt^r$ the primitive subgroups of $\ZM^r$ and by  $L_0(\RM^r, \RM^{n+1})$ the vector space of rank $r$ linear mappings from $\RM^r$ to $\RM^{n+1}$. 
  \begin{theorem}
 \label{T::KM} Let $h:\RM^d \supset B(0,3R) \to L_0(\RM^r,\RM^{n+1})$ be such that for any $\G \in \Lt^r$  the mapping
 $$\psi_\G:B(0,3^rR) \to \RM,\ x \mapsto \| h(x)\G \|$$ is of  $(C,\tau)$ class. Assume that there exists $\rho \leq 1$ such that the inequality
   $$\| \psi_\G \|_{B(0,R)} \geq \rho$$
   holds for any $\G \in \Lt^r$.
 There exists a  constant $C'$ which depends only on $C,d,r$ such that
$$\Vol(\{ x \in B(0,R): \dt(h(x)\ZM^r) \leq \e \}) \leq C' \left( \frac{\e}{\rho}\right)^\tau R^d $$
for any $\e \leq \rho$.
\end{theorem}
The value of $C'$ is given in the paper of Kleinbock and Margulis who proved the theorem for the case $r=n+1$~\cite[Theorem 5.2]{Kleinbock_Margulis}.
Theorem~\ref{T::KM} is given by Kleinbock and the proof is essentially the same as for $r=n+1$~\cite[Theorem 2.6]{Kleinbock_Baker}. There also exists a more general statement due to Bernick, Kleinbock and Margulis~\cite[ Theorem 6.2]{Bernick_Kleinbock_Margulis}. 
\section{Discrete subgroups and arithmetic classes}
 We now adapt the Kleinbock-Margulis method to our setting~\cite{Kleinbock_Margulis}.
To the vector $\a \in \RM^n$, we associate, with Schmidt, the discrete subgroup $[\a]$  in $\RM^{n+1}$ of rank $n$ defined by
$$[\a]:=\{(i,(\a,i)) \in \RM^{n+1}: i\in \ZM^n \} $$
where $(\cdot,\cdot)$ denotes the euclidean scalar product~\cite{Schmidt}. 
Consider the linear map
$$g_t:\RM^{n+1} \to \RM^{n+1} $$ 
whose matrix in the standard basis is diagonal with coefficients~:
$$(e^{-t},e^{-t},\dots,e^{-t},e^{nt}). $$

Given a discrete subgroup $\G \subset \RM^{n+1}$, we use the notation
$$\dt(\G):=\inf_{\g \in \G \setminus \{ 0 \}} \| \g \| $$
where $\| \cdot \|$ denotes the euclidean norm. 

\begin{lemma}
\label{L::arithmetique} Let $i \in \ZM^n$ be such that $|(\a,i)| \leq a $, $a\neq 0$, then
$$\dt(g_t [\a]) \leq \e $$ where $\e,t$ are defined by
$$\left\{ \begin{matrix}\e&=&\sqrt{2} \left(a\| i \|^n\right)^{\frac{1}{n+1}} ; \\ t &=&\frac{1}{n+1}\log \frac{\| i \|}{a}   \end{matrix} \right. $$  
\end{lemma}
\begin{proof}
For any $x \in \RM^n$ and any $y \in \RM$, we have~:
$$\| (x,y) \| \leq \sqrt{2} \max\left( \| x \| ,| y | \right) .$$
Consequently, the estimates $|(\a,i)| \leq a $ gives~:
$$\| g_t (i,(\a,i) \| \leq  \sqrt{2} \max\left(e^{-t} \| i \| , e^{nt} a \right)=\e $$
\end{proof} 

\section{Proof of Proposition~\ref{P::KM}}
Consider the vector subspace $V \subset \RM^n$  inside which the map 
$$f:U \to \RM^n,\ f(U) \subset f(0)+V$$
is non-degenerate and let us denote by 
$e_1,e_2,\dots,e_k$ a basis of the vector space $V$. We embed $\RM^n$ in $\RM^{n+1}$ using the map
$$\RM^n \to \RM^{n+1},\ x \mapsto (x,0) $$ and put $e_{n+1}=(0,\dots,0,1)$.

Denote by $L(-,-)$ the space of linear mappings and consider the $t$-dependent mappings
 $$h_t:V \to L(V,\RM^{n+1})$$
 defined by
$$h_t(x):V \to \RM^{n+1},\
e_i \mapsto   e^{-t}e_i+e^{nt} (e_i,f(x))e_{n+1}  $$
 In particular, if $\G$ is the subgroup generated by $e_1,\ e_2,\dots,\ e_k$, then $h_0(x) \G$ is the discrete group associated to $f(x)$~:
 $$h_0(x)\G=[f(x)] .$$
 We denote by 
 $$\pi_\G:V \to V  $$
 the orthogonal projection on the vector space $\G \otimes \RM$.
 
\begin{lemma}
\label{L::norm}
For any discrete subgroup $\G \subset \RM^n$ of rank $r$, we have 
$$\| h_t(x)\G \|=   \sqrt{e^{-2rt}+ e^{2(n-r+1)t}\| \pi_\G(f(x))\|^2}\, \| \G \| .$$
\end{lemma}
\begin{proof}
Choose a basis $u_ i=1,\dots,r$ of $\G$. The norm of $h_t(x)\G$ is the norm of the vector
$$e^{-rt} u_1 \w \dots \w u_r+e^{(n-r+1)t}\sum_{i=1}^r (-1)^i (u_i,f(x)) u_1\w \dots \w u_{i-1} \w \hat u_i \w u_{i+1} \dots \w u_r \w e_{n+1}. $$

Choose orthonormal vectors $b_1,\dots,b_r$ of $\RM^{n+1}$ which span the same $r$-dimensional vector space as $u_1,\dots,u_r$
and which define the same orientation as $u_1,\dots,u_r$, that is~:
$$u_1 \w u_2 \w \dots \w u_r= \| \G \|  b_1\w b_2 \w  \dots \w b_r $$
I assert that for any vector $\a \in \RM^n$, the polyvectors
$$A(x):=\sum_{i=1}^r (-1)^i (u_i,\a) u_1\w \dots \w u_{i-1} \w \hat u_i \w u_{i+1} \dots \w u_r $$
and
$$B(x):=   \| \G \|  \sum_{i =1}^r (-1)^i(b_i,\a) b_1\w \dots \w b_{i-1} \w \hat b_i \w b_{i+1} \dots \w b_r$$
are equal. Indeed for any $i=1,\dots,r$, we have
$$u_i \w A(x) =(u_i,\a)   \| \G \|\, b_1\w  \dots \w b_r=u_i \w B(x) .$$
This proves the assertion. Let us now apply this assertion to the vector
$$\a=e^{(n-r+1)t} f(x) . $$ As the vectors $b_1,\dots,b_r$ are orthonormal, we get
$$\| A(x) \|=\| \G \|\, \| \pi_\G(e^{(n-r+1)t} f(x)\|. $$
This concludes the proof of the lemma.
\end{proof}
Combining the lemma with the estimate
$$ \sqrt{x_1^2+\dots+x_n^2} \geq \frac{1}{\sqrt{n}} \left(|x_1|+|x_2|+\dots+|x_n|\right)$$
we deduce the existence, for any discrete subgroup~$\G$, of constants $c_1,\dots,c_n \in \RM_+$ such that ~:
$$ \| h_t(x)\G \| \geq c_0+c_1 |f_1|+\dots+c_n |f_n| $$
Thus by Corollary \ref{C::KM}, there exist $C,\tau$ such that,  for any discrete subgroup~$\G$, the restriction of the  function 
$$x \mapsto \frac{1}{\| \G \|}\| h_t(x)\G \|=\sqrt{ e^{-2rt}+ \|\pi_\G f(x)\|^2}  $$
to an appropriate neighbourhood of the origin is of class $(C,\tau)$.
As the $(C,\tau)$-class functions define a cone, the maps
$$x \mapsto \| h_t(x)\G \| $$
are also of class $(C,\tau)$.   

To apply the  Kleinbock-Margulis theorem, we need to minorate the norms of the lattices $h_t(x)\G$.
\begin{lemma} There exists constants $A,r_0$ such that for any lattice $\G \subset \RM^n$ and for any $r<r_0$, we have
$$\sup_{x \in B(0,r)}\| h_t(x)\G \| \geq Ar^l  . $$
\end{lemma}
\begin{proof}
Without loss of generality, we may assume that $R \leq 1/3$.
Denote respectively by $S_\G \subset V$ and $S_r $ the unit spheres in $\G \otimes \RM$ and in $\RM^n$. We have the estimates
$$\| h_t(x)\G \| \geq \| \pi_\G(f(x))\| \geq \sup_{(v,x) \in S_\G \times S_r}  (v,f(x)) $$
 Fix $(v_0,x_0) \in S_\G \times S_R$ for which $  (v,f(x))$ gets its maximum.
 
As $f$ is $C^l$, the Taylor expansion at the origin of the function
$$x \mapsto (v_0,f(x))$$
is of the form
$$(v_0,f(x))=\sum_{| I | \leq l }a_I(v)x^I+o(\| x \|^l) $$
Note that all $a_I(v)$'s cannot be equal to zero since $f$ is $l$-curved.
For any $\e$, there exists a constant $r_0$ such that
$$| (v_0,f(x))-\sum_{| I | \leq l}a_I(v_0)x^I | \leq   (v_0,f(x_0))\| x \|^l  $$ 
for  any $x \in B(0,r_0)$. 
We have the estimates
$$(v_0,f(tx_0)) \geq \sum_{| I | \leq l }t^{|I| }a_I(v)x_0^I -  (v_0,f(x_0)) \| x_0\|^l t^l \geq  \left( 1-2\| x_0\|^l\right)    (v_0,f(x_0))  t^l$$
and as $R<1/3$ this gives
$$(v_0,f(tx_0)) \geq \frac{1}{3} (v_0,f(x_0))  t^l$$
Now define 
$$ r:=\frac{t}{\| x_0 \|},\ x_r:=r \| x_0 \| \in B(0,r) .$$
We have 
$$\sup_{x \in B(0,r)}\| h_t(x)\G \| \geq   (v_0,f(x_r)) \geq   \frac{\| x_0 \|^l  (v_0,f(x_0))}{3}  r^l .  $$
This concludes the proof of the lemma.
 \end{proof}

We now apply the Kleinbock-Margulis theorem~:
 \begin{proposition}
    For any $l$-curved mapping
   $$f=(f_1,\dots,f_n):B(0,3R) \to \RM^n, f(0) \neq 0  $$
    there exist constants $A,C,r_0 >0$ such that
 $$\Vol(\{ x \in B(0,r): \dt([g_tf(x)]) \leq \e \}) \leq C \left(\frac{\e}{r^l}\right)^{1/dl} r^{d}$$
 for any $r \leq r_0$, any $\e \leq Ar^l$ and any $t \geq 0$.
 \end{proposition}
 In the context of Lemma \ref{L::arithmetique}, we have~:
 $$\left\{ \begin{matrix}\e&=&\sqrt{2} \left(a\| i \|^n\right)^{\frac{1}{n+1}} ; \\ t &=&\frac{1}{n+1}\log \frac{\| i \|}{a}   \end{matrix} \right. $$
 thus according to the proposition above, there exist  constants $A',C'$ such that
 $$\Vol(\{ x \in B(0,r): |(f(x),i)| \leq a \}) \leq C' \left(a \| i \|^n \right)^{\frac{1}{dl(n+1)}} r^{-1/d}\Vol(B(0,r)) $$
 provided that 
 $$\left\{ \begin{matrix}\left(a\| i \|^n\right)^{\frac{1}{n+1}}& \leq &A'r^l\ ; \\  a & \leq & \| i \|.\end{matrix} \right. $$
 This proves Proposition~\ref{P::KM}.
\section{Proof of Theorem \ref{T::arithmetique}}
Denote by $[\cdot ]$ the integer value and consider the map
$$\p:\ZM^n \to \NM,\ i \mapsto \left[{\log}_2 \| i \| \right]+1. $$  
For $i \in \ZM^n$, $\p(i)$ is the smallest natural number such that $i$ is contained in the ball of radius $2^{\p(i)}$ centred at origin.

Define the sequence $\rho=(\rho_k)$ by
$$\rho_k:= 2^{-kn-kl(n+1)-(n+1)^2 kdl}a_k^{(n+1)l-1} $$
so that $a'_k=\rho_k a_k.$

Fix $i \in \ZM^n $ and put $k:=\p(i)$. The set
$$ M_i:=\{\b \in  \RM^{n}: \left| (\b,i) \right| <  \rho_k a_k \} $$
is a band of width $2\rho_k a_k/\| i \|$ and the union over the $i$'s of the subsets $M_i$ is the complement of the arithmetic class $\Ct(\rho a)$~:
$$\RM^n \setminus \Ct(\rho a)=\bigcup_{i \in \ZM^n} M_i .$$

\begin{figure}[ht]
\centerline{\epsfig{figure=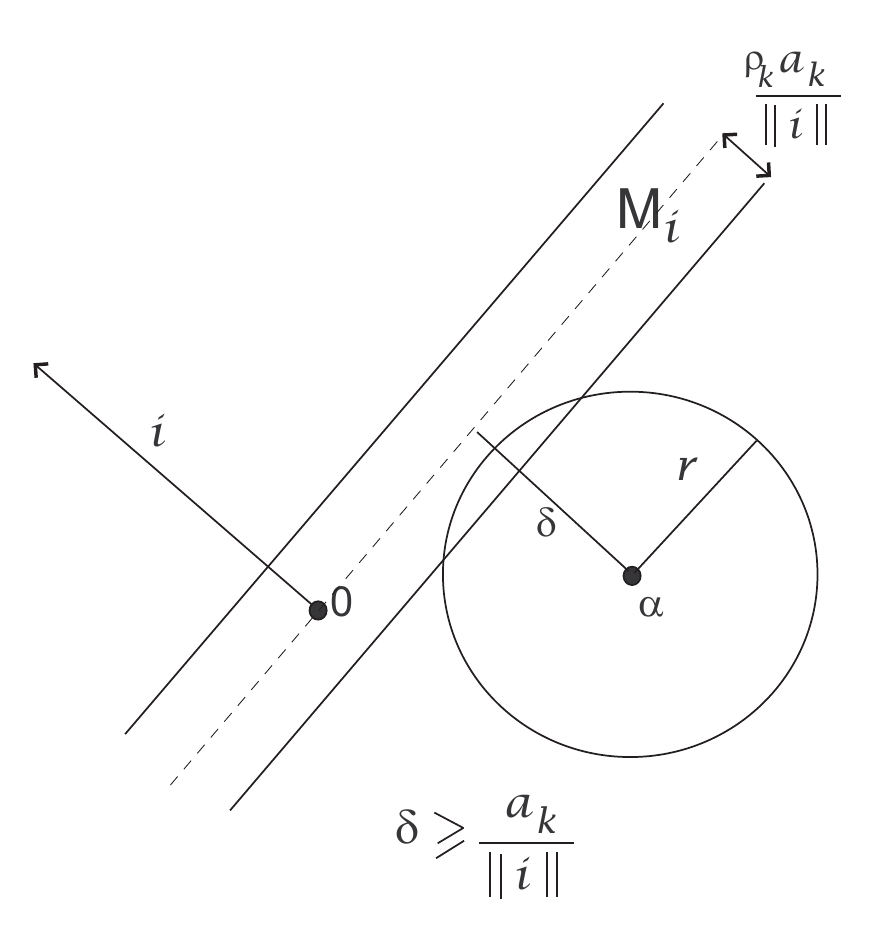,height=0.52\linewidth,width=0.48\linewidth}}
\end{figure} 

 Let $\a$ be a vector in  $ \Ct(a)$. Denote by $\dt_i$ the distance from $\a$ to the hyperplane orthogonal to the vector $i \in \ZM^n$.
 The intersection of the set $M_i$ with the ball $B(\a,r)$ can possibly be non-empty only if
  $$r > \dt_i-\frac{\rho_k a_k}{\| i \|}.  $$
 As $\a \in \Ct(a)$, we have
 $$|(\a,i)| \geq a_k $$
 thus
 $$\dt_i \geq \frac{a_k}{\| i \|} \geq \frac{a_k}{2^k}  $$
 and therefore
$$ \frac{(1-\rho_k)a_k}{2^k} < r.$$

As the sequence $\rho$ is summable, there exists an integer $N$ such that
$$\rho_k<\frac{1}{2},\ \forall k \geq N. $$
Choose
$$r <\inf \lb \frac{(1-\rho_k)a_k}{2^k}: k \leq N \rb$$
then 
$$ \p(i)<N \implies M_i \cap B(\a,r)=\emptyset.$$ 

This shows that if $M_i$ intersects the ball $B(\a,r)$ then the vector $i \in \ZM^n$  belongs to the set
 $$I_r:=\{ i \in \ZM^n: \frac{a_k}{2^{k+1}} < r,\ k :=\p(i) \}. $$
 In particular, when $r$ tends to zero, the elements $i \in I_r$ define values of $\p(i)$ which go to infinity.
 
As the mapping $f$ is $l$-curved, according to Proposition \ref{P::KM}, there exist constants $A,C$ such that
  $$\Vol (B(0,r) \cap f^{-1} (M_i)) \leq C( a_k \rho_k \| i \|^n r^{-l(n+1)})^{\frac{1}{dl(n+1)}}  \Vol(B(0,r)), $$
  provided that
  $$\left\{ \begin{matrix} \left(\rho_k a_k \| i \|^n\right)^{\frac{1}{n+1}}& \leq &Ar^l\ ; \\  \rho_k a_k & \leq & \| i \|.\end{matrix} \right. $$
  Without loss of generality, we may assume that the sequence $a$ is bounded from above by $1$. In that case, by definition of $\rho_k$, the second condition is automatically satisfied.  I assert that, for $r$ small enough, the first condition is also fulfilled. 
  As 
  $$\rho_k=2^{-kn-kl(n+1)-(n+1)^2 kdl}a_k^{(n+1)l-1} ,$$
  we have
 $$\left( \rho_k a_k \| i \|^n\right)^{\frac{1}{n+1}} \leq 2^{-kdl(n+1)} a_k^l,$$
 with $k=\p(i)$.  As the vector $i \in \ZM^n$ belongs to $I_r$, we get that 
 $$ 2^{-kdl(n+1)} a_k^l \leq  2^{-kl(d(n+1)-1)+l} r^l \leq  2^{-kl+l} r^l .$$
 The coefficient behind $r^l$ goes to zero as $k$ goes to infinity, this proves the assertion.
 
 As $i \in \ZM^n$ belongs to $I_r$, we have
 $$r^{-l(n+1)} \leq 2^{(k+1)l(n+1)}a_k^{-l(n+1)}, $$
 with $k=\p(i)$. Thus, we get the estimate
 $$\Vol (B(0,r) \cap f^{-1} (M_i)) \leq C(2^{kn+kl(n+1)+l(n+1)} a_k^{1-l(n+1)} \rho_k)^{\frac{1}{dl(n+1)}} \Vol(B(0,r)) $$
 for $r$ small enough. Using the definition of $\rho$, the expression between brackets can be simplified, namely~:
 $$ (2^{kn+kl(n+1)+l(n+1)} a_k^{1-l(n+1)} \rho_k)^{\frac{1}{dl(n+1)}}= 2^{1/d} 2^{-k(n+1)} $$
 As the map
 $$f:\RM^d\to \RM^n $$
is differentiable, by the mean-value theorem,  there exists a constant $\kappa$ such that
 for any sufficiently small $r$
 $$f(B(0,r)) \subset B(\a,{\kappa r}),\ f(0)=\a. $$
 In particular
  $$f^{-1} (M_i) \cap B(0,r) \neq \emptyset \implies i \in I_{\kappa r},$$
  for any sufficiently small $r$.
 
 This shows that the measure of the complement to   $f^{-1}(\Ct(\rho a))$ in $B(0,r)$ is bounded from above by
 $$C2^{1/d} \Vol(B(0,r))\sum_{i \in I_{\kappa r}}2^{-\p(i)(n+1)} .$$
 
 We have
$$\#\{ \p(i)=k \}=\#\{ \p(i)\leq k\}-\#\{ \p(i)\leq k-1\} \leq 2^{(k+1)n} $$
where the symbol $\#-$ stands for the cardinal. This shows that
$$\sum_{i \in \ZM^n}2^{-\p(i)(n+1)}  \leq    \sum_{k \geq 0}2^{(k+1)n} 2^{-k(n+1)}=\frac{2^n}{1-2^{-1}},$$
therefore the sums
$$\sum_{i \in I_{\kappa r}}2^{-\p(i)(n+1)} $$
converge to $0$ as $r$ tends to zero. This concludes the proof of Theorem~\ref{T::arithmetique}. 

   \noindent {\bf Acknowledgements.} { The idea of arithmetic density originated from a short discussion with J.-C. Yoccoz which I thank sincerely. Many thanks to D. Kleinbock,  and B. Weiss for explanations on diophantine approximation, to F. Jamet, Remarque and R. Uribe for remarks concerning measure theory, to F. Aicardi  for the picture which illustrate the proof of the theorem and to Duco van Straten for suggestions on the writing of this paper. Thanks also to B. Fayad, who pointed out to me a mistake in the original version of the paper.
    
This research is supported by the  Max Planck Institut für Mathematik in Bonn and by the Deutsche Forschungsgemeinschaft project, SFB-TR 45, M086, {\em Lagrangian geometry of integrable systems.}}

 \bibliographystyle{amsplain}
\bibliography{master}

\providecommand{\bysame}{\leavevmode\hbox to3em{\hrulefill}\thinspace}
\providecommand{\MR}{\relax\ifhmode\unskip\space\fi MR }
\providecommand{\MRhref}[2]{%
  \href{http://www.ams.org/mathscinet-getitem?mr=#1}{#2}
}
\providecommand{\href}[2]{#2}
\begin{thebibliography}{10}

\bibitem{Arnold_KAM}
V.I. Arnold, \emph{{ Proof of a theorem of A. N. Kolmogorov on the preservation
  of conditionally periodic motions under a small perturbation of the
  hamiltonian}}, Uspehi Mat. Nauk \textbf{18} (1963), no.~5, 13--40, English
  translation: Russian Math. Surveys.

\bibitem{Bernick_Kleinbock_Margulis}
V.~Bernick, D.~Kleinbock, and G.A. Margulis, \emph{Khintchine-type theorems on
  manifolds~: the convergence case for standard and multiplicative versions},
  Int. Math. Research Notices \textbf{9} (2001), 453--486.

\bibitem{Cassels}
J.W.S. Cassels, \emph{An introduction to diophantine approximation}, Cambridge
  Univ. Press, 1957.

\bibitem{Dani}
S.G. Dani, \emph{Divergent trajectories of flows on homogeneous spaces and
  diophantine approximation}, J. Reine Angew. Math. 359 \textbf{6} (1985),
  no.~2, 55--89.

\bibitem{Kleinbock}
D.~Kleinbock, \emph{Extremal subspaces and their submanifolds}, Geom. Funct.
  Anal \textbf{13} (2003), no.~2, 437--466.

\bibitem{Kleinbock_Baker}
\bysame, \emph{{Baker--Sprind\v{z}huk conjectures for complex analytic
  manifolds}}, Algebraic groups and arithmetic. Proceedings of the
  International Conference held in Mumbai, December 17--22, 2001. (S.~G. Dani
  and Gopal Prasad, eds.), Tata Institute of Fundamental Research Studies in
  Mathematics Series, Narosa Publishing House, 2004, pp.~539--553.

\bibitem{Kleinbock_dichotomy}
\bysame, \emph{An `almost all versus no' dichotomy in homogeneous dynamics and
  diophantine approximation}, Geom. Dedicata \textbf{149} (2010), 205--218.

\bibitem{Kleinbock_Margulis}
D.Y. Kleinbock and G.A. Margulis, \emph{Flows on homogeneous spaces and
  diophantine approximation on manifolds}, Ann. of Math. \textbf{148} (1998),
  339--360.

\bibitem{Kolmogorov_KAM}
A.~N. Kolmogorov, \emph{On the conservation of quasi-periodic motions for a
  small perturbation of the hamiltonian function}, Dokl. Akad. Nauk SSSR
  \textbf{98} (1954), 527--530.

\bibitem{Lebesgue_derivation}
H.~Lebesgue, \emph{Sur l'existence des d\'eriv\'ees}, Comptes rendus \`a
  l'Acad\'emie des sciences \textbf{136} (1903), 659--661.

\bibitem{Dirichlet}
G.~Lejeune-Dirichlet, \emph{{Verallgemeinerung eines Satzes aus der Lehre von
  der Kettenbrüchen nebst einigen Anwendungen auf die Theorie der Zahlen}},
  Werke, Band I (1842), 633--638.

\bibitem{Margulis}
G.A. Margulis, \emph{On the action of unipotent groups in the space of
  lattices}, Proc. of the Summer School on group representations, Bolyai Janos
  Math. Soc., Budapest, 1971, pp.~365--371.

\bibitem{Moser_KAM}
J.~Moser, \emph{{On the construction of almost periodic solutions for ordinary
  differential equations (Tokyo, 1969)}}, Proc. Internat. Conf. on Functional
  Analysis and Related Topics, Univ. of Tokyo Press, 1969, pp.~60--67.

\bibitem{Pyartli}
A.S. Pyartli, \emph{Diophantine approximations on submanifolds of euclidean
  space}, Functional Analysis and Its Applications \textbf{3} (1969), no.~4,
  303--306.

\bibitem{Russmann_KAM}
H.~R\"ussmann, \emph{Nondegeneracy in the perturbation theory of integrable
  dynamical systems}, Number theory and dynamical systems (York, 1987), London
  Math. Soc., Cambridge University Press, 1989, pp.~5--18.

\bibitem{Schmidt}
W.~Schmidt, \emph{{Diophantine approximation and certain sequences of
  lattices}}, Acta Arith. \textbf{18} (1940), 165--178.

\end{thebibliography}
 \end{document}